\documentclass[11pt,reqno]{amsart}

\usepackage{amsmath,amsthm,amsfonts,amssymb}
\numberwithin{equation}{section}

\usepackage{verbatim} 
\usepackage{xspace} 
\usepackage{amssymb}
\usepackage{amsmath}
\usepackage{amsthm} 
\usepackage[latin1]{inputenc}
\usepackage{graphicx}
\usepackage{amscd,amssymb,euscript,graphics}

\makeindex

\newcommand{\set}[1]{\left\{#1\right\}}
\newcommand{\card}[1]{\left|#1\right|}
\newcommand{\floor}[1]{\left\lfloor#1\right\rfloor}

\newcommand{\mP}{\mathcal P}
\newcommand{\sfL}{\mathsf L}
\newcommand{\sfR}{\mathsf R}

\newcommand{\ptn}{\mathsf {ptn}}
\newcommand{\len}{\mathsf {len}}

\renewcommand{\vec}[1]{\overrightarrow{#1}}

\newcommand{\Z}{\ensuremath{\mathbb{Z}}}

\newtheorem{theorem}{Theorem}[section]
\newtheorem{teo}{Theorem}[section]
\newtheorem{prop}[teo]{Proposition}

\newtheorem{lemma}[teo]{Lemma}

\newtheorem{df}[teo]{Definition}

\newtheorem{example}[teo]{Example}

\begin{document}

\title[Expanders, Generalized Hyperbolic Spaces and Product of Trees]{Traffic Congestion in Expanders, $(p,\delta)$--Hyperbolic Spaces and Product of Trees}

\author[Shi Li and Gabriel H. Tucci]
{Shi Li and Gabriel H. Tucci}

\address{Bell Laboratories Alcatel--Lucent,
Murray Hill, NJ 07974, USA}

\email{gabriel.tucci@alcatel-lucent.com}
\email{shili@princeton.edu}

\begin{abstract}
In this paper we define the notion of $(p,\delta)$--Gromov hyperbolic space where we relax Gromov's {\it slimness} condition to allow that not all but a positive fraction of all triangles are $\delta$--slim. Furthermore, we study maximum vertex congestion under geodesic routing and show that it scales as $\Omega(p^2n^2/D_n^2)$ where $D_n$ is the diameter of the graph. We also construct a constant degree family of expanders with congestion $\Theta(n^2)$ in contrast with random regular graphs that have congestion $O(n\log^{3}(n))$. Finally, we study traffic congestion on graphs defined as product of trees.
\end{abstract}

\maketitle

\section{Introduction} 

The purpose of this work is to continue the study of traffic congestion under geodesic routing. By geodesic routing we mean that the path chosen to route the traffic between the nodes is the minimum length path. If there are several paths with the shortest distance (i.e. several geodesics joining the nodes) then we divide the traffic equally among these. We are restricting ourselves to connected graphs since traffic in non--connected graphs needs to be analyzed in a component by component basis. Our set up through out the paper is the following. Let $\{G_{n}\}_{n=1}^{\infty}$ be a family of connected graphs where $G_{n}$ has $n$ nodes. For each pair of nodes in $G_{n}$, consider a unit flow of traffic that travels through the minimum path(s) between the nodes as we previously discussed. Hence, the total traffic flow in $G_{n}$ is equal to $n(n-1)/2$. Given a node $v\in G_{n}$ we define $\mathcal{L}_{n}(v)$ as the total traffic flow passing through the node $v$. Let $M_{n}$ be the maximum vertex flow across the graph
$$
M_{n}:=\max \Big\{\mathcal{L}_{n}(v)\,:\,v\in G_{n}\Big\}.
$$
It is easy to see that for any graph $n-1\leq M_{n}\leq n(n-1)/2$. Throughout the paper we use the following standard notation. For two positive functions $f$ and $g$ we say that $f=O(g)$ if there exists a constant $k$ such that $f(n)\leq kg(n)$. We denote $f=\Theta(g)$ if there exist constants $k_1$ and $k_2$ such that $k_1 g(n)\leq f(n)\leq k_2g(n)$ and finally $f=\Omega(g)$ is there exists $k$ such that $f(n)\geq kg(n)$.

It was observed experimentally in \cite{NS1}, and proved formally in \cite{BT1, BT2}, that if the family is Gromov hyperbolic then the maximum vertex congestion scales as $\Theta(n^2)$. More precisely, let $\{G_{n}\}_{n=1}^{\infty}$ be an increasing sequence of finite simple graphs that is uniformly Gromov $\delta$--hyperbolic, for some non--negative $\delta$, then there is a sequence of nodes $\{x_{n}\}_{n=1}^{\infty}$ with $x_{n}\in G_{n}$ such that the total traffic passing through $x_{n}$ is greater than $cn^2$ for some positive constant $c$ independent on $n$. These highly congested nodes are called the {\it core} of the graph.  
 
In this work we extend this analysis and study what happens to the traffic congestion if we relax the slimness condition so that not all but a fraction of all triangles are $\delta$--slim. More precisely, we say that a metric $(X,d)$ is $(p,\delta)$--hyperbolic if for at least a $p$ fraction of the 3--tuples $(u,v,w) \in X^3$ the geodesic triangle $\triangle_{uvw}$ is $\delta$--slim. The case $p$ equals to $1$ corresponds to the classic Gromov $\delta$--hyperbolic spaces. We show that congestion in these graphs scales as $\Omega(p^2n^2/D_n^2)$ where $D_n$ is the diameter of $G_{n}$.

It was shown in \cite{Buyalo} that hyperbolic groups embed quasi--isometrically in the product of finitely many regular trees. For this reason, we believe that it is interesting to understand the traffic characteristic in these graphs as well. Let $T_{d}^{k}$ be the infinite graph defined as the product $T_{d}\times\cdots\times T_{d}$ of $k$ infinite regular trees of degree $d$. Let us fix an arbitrary vertex $v^*$ in $T_{d}^{k}$ and let $G_{n}$ be the induced graph by the ball of centrer $v^*$ and radius $r$ in the underlying graph. Let $n=n(r)$ be the cardinality of $B(v^*,r)$. We show that the maximum vertex congestion in $G_{n}$ scales as $\Theta(n^2/\log_{d-1}^{k-1}(n))$ as $r$ increases.

Another important family of graphs are expander. In graph theory, an expander graph is a sparse graph that has strong connectivity properties. Expander constructions have spawned research in pure and applied mathematics, with several applications to complexity theory, design of robust computer networks, and the theory of error--correcting codes. It is well known that random regular graphs are a large family of expanders. It was proved in \cite{T} that random $d$--regular graphs have maximum vertex congestion scaling as $O(n\log_{d-1}^{3}(n))$. Therefore, it is a natural question to ask if expanders have always low congestion under geodesic routing. In Section \ref{expanders}, we show that this is not the case. More precisely, we construct a family of expanders $\{G_{n}\}_{n=1}^{\infty}$ with maximum vertex congestion $\Theta(n^2)$.

\section{Preliminaries}\label{prelim}
\subsection{Hyperbolic Metric Spaces}
\noindent In this Section we review the notion of Gromov $\delta$--hyperbolic space. There are many equivalent definitions of Gromov hyperbolicity but the one we take as our definition is the property that triangles are {\it slim}.

\begin{df}
Let $\delta>0$. A geodesic triangle in a metric space $X$ is said to be $\delta$--slim if each of its sides is contained in the $\delta$--neighbourhood of the union of the other two sides. A geodesic space $X$ is said to be $\delta$--hyperbolic if every triangle in $X$ is $\delta$--slim. 
\end{df}

\noindent It is easy to see that any tree is $0$-hyperbolic. Other examples of hyperbolic spaces include, the fundamental group of a surface of genus greater or equal than 2, the classical hyperbolic space, and any regular tessellation of the hyperbolic space (i.e. infinite planar graphs with uniform degree $q$ and $p$--gons as faces with $(p-2)(q-2)>4$).

\subsection{Definition of the Core}\label{core}

\noindent We review the definition of the core of a graph given in \cite{BT1}. Let $\{G_{n}\}_{n=1}^{\infty}$ be a sequence of finite and simple graphs such that $|G_{n}|=n$.

\noindent For each fixed $n$ consider a load measure $\mu_{n}$ supported on $G_{n}$. This measure defines a traffic flow between the elements in $G_{n}$ in the following way; given $u$ and $v$ nodes in $G_{n}$ the traffic between these two nodes splits equally between the geodesics joining $u$ and $v$, and is equal to $\mu_{n}(v)\mu_{n}(u)$. Note that the uniform measure determines a uniform traffic between the nodes. Given a subset $A\subset G_{n}$ we denote by $\mathcal{L}_{n}(A)$ the total traffic passing through the set $A$. The total traffic in $G_{n}$ is equal to $\mathcal{L}_{n}(G_{n})$.

Given a node $y\in G_{n}$ and $r>0$ we denote by $B(y,r)$ the ball of center $y$ and radius $r$.

\begin{df}
Let $\alpha$ be a number between $0$ and $1$. We say that a point $y\in G_n$ is in the asymptotic $\alpha$--core if 
\begin{equation}
ap(y,r_\alpha):=\liminf_{n\to\infty}{\frac{\mathcal{L}_{n}(B(y,r_{\alpha}))}{\mathcal{L}_{n}(G_{n})}}\geq \alpha
\end{equation}
for some $r_{\alpha}$ independent on $n$. The set of nodes in the asymptotic $\alpha$--core is denoted by $C_{\alpha}$. The core of the graph is the union of all the $\alpha$--cores for all the values of $\alpha$. This set is denoted by $\mathcal{C}$
\begin{equation}
\mathcal{C} = \cup_{\alpha>0}{\,C_\alpha}.
\end{equation}
\end{df}
\noindent We say that a graph has a core if $\mathcal{C}$ is non--empty.

Roughly speaking a point $y$ belongs to the $\alpha$--core if there exists a finite radius $r$, independent on $n$, such that the proportion of the total traffic passing through the ball of center $y$ and radius $r$ behaves asymptotically as $\alpha n^{2}/2$ as $n\to\infty$.

\vspace{0.2cm}
\begin{example}
Let $d\geq 2$ and let $T_d$ be the associated $d$--regular tree. Consider the increasing sequence of sets $\{T_{d}(n)\}_{n=1}^{\infty}$ where $T_{d}(n)$ is the truncated tree at depth $n$, and $\mu_{n}$ is uniform measure on $T_{d}(n)$. It was shown in \cite{BT1} that the asymptotic proportion of traffic through the root is $ap(\mathrm{root})=1-\frac{1}{d}$. Therefore, this graph has a non-empty core. Moreover, the core is the whole graph! However, for every $\alpha$ the $\alpha$--core is finite.
\end{example}

\vspace{0.2cm}
\begin{example}
Let $G=\Z^{p}$ and let $G_{n}$ be the graph induced by the set $\{-n/2,\ldots,n/2\}^{p}$ with the uniform measure. A simple but lengthy analytic calculation shows us that for every $x\in G$ the traffic flow through $v$ behaves as $\Theta(n^{1+\frac{1}{p}})$. Hence, if $p\geq 2$ the core of this graph is empty. This is not the case for $p=1$.
\end{example}

\subsection{Properties of the Core for $\delta$--Hyperbolic Graphs}\label{hyper}
  
Let  $\{G_{n}\}_{n=1}^{\infty}$ be a sequence of finite graphs as in the previous section. We are interested in understanding the asymptotic traffic flow through an element in the graph as $n$ grows. More precisely, for each $v$ element in $G_n$ and $r>0$ consider $B(v,r)$ the ball centred at $v$ and radius $r$. Let $\mathcal{L}_{n}(v,r)$ be traffic flow passing through this ball in the graph $G_{n}$. We want to study $\mathcal{L}_{n}(v,r)$ as $n$ goes to infinity. As we define in Section \ref{core} an element $v$ is in the asymptotic $\alpha$--core if 
$$
ap(B(v,r_\alpha))=\liminf_{n\to\infty}{\frac{\mathcal{L}_{n}(v,r_{\alpha})}{\mathcal{L}_{n}(G_{n})}}\geq \alpha
$$
for some radius $r_{\alpha}$ independent on $n$. In other words, asymptotically a fraction $\alpha$ of the total traffic passes within bounded distance of $v$.

\begin{teo}[Baryshnikov, T.] \label{main1}
Assume that there exists a non--negative constant $\delta$ such that $G_{n}$ is $\delta$--hyperbolic for every $n$. Then there exists a sequence $\{x_{n}\}_{n=1}^{\infty}$ with $x_{n}\in G_{n}$ and a positive constant $c$ independent on $n$ such that 
$$
\mathcal{L}_{n}(x_{n})\geq cn^2.
$$ 
\end{teo}

\subsection{Expanders} We say that a family of graphs $\{G_{n}\}_{n=1}^{\infty}$ is a $c$--expander family if the edge expansion (also isoperimetric number or Cheeger constant) $h(G_n)\geq c$ where  
$$
h(G_n) = \min \Bigg\{\frac{|\partial S|}{|S|} \,\,:\,\,S\subset G_{n} \,\,\text{with}\,\, 1\leq |S|\leq |G_{n}|/2 \Bigg\}
$$
and $\partial S$ is the edge boundary of $S$, i.e., the set of edges with exactly one endpoint in $S$.

\section{Congestion on $(p,\delta)$--Hyperbolic Graphs}

In this Section we generalize the definition of Gromov hyperbolic spaces to include spaces where not all but a fixed proportion of the triangles in the metric space are {\it slim}. Furthermore, we study their traffic characteristics under geodesic routing. 
 
\begin{df}[$(p, \delta)$--Hyperbolic]
We say that a metric $(X,d)$ is $(p,\delta)$--hyperbolic if for at least a proportion $p$ of all geodesic triangles in $X$ are $\delta$--slim.
\end{df}

The classical Gromov $\delta$--hyperbolic spaces correspond to $p$ equals one.

\begin{theorem}
Let $(X,d)$ be a $(p,\delta)$--hyperbolic metric space of size $n$.  Let $D$ be its diameter and $M = \max\set{\card{B(u,\delta)} : u \in X}$ be the maximum number of points in a ball of radius $\delta$.  Then there exists a point $a \in X$ with congestion at least $p^2n^2/(D^2M^3)$. 
\end{theorem}

Before proving the theorem, we prove the following useful lemma.
\begin{lemma}\label{LL}
Let $G = (U, V, E)$ be a bipartite graph such that $|U| = |V| = n$ and $|E| \geq pn^2$. The edges of $G$ are colored in such a way that every vertex in $U \cup V$ is incident to at most $t$ colors ($u$ is incident to a color $c$ if $u$ is incident to an edge with color $c$).  Then, there exists a color that is used by at least $(pn/t)^2$ edges in $E$. 
\end{lemma}

\begin{proof}
Define three random variables $A,B$ and $C$ as follows. We randomly select an edge $(u,v) \in E$ and let  $A=u, B=v$ and $C$ be the color of the edge $(u,v)$. 

Since $h(A) \leq \log n$, $h(B) \leq \log n$ and $h(A, B) \geq \log (pn^2) = 2\log n - \log(1/p)$ we have that 
$$
I(A;B) = h(A) + h(B) - h(A,B) \leq \log (1/p)$$
where $h$ is the entropy function and $I(A;B)$ is the mutual information between $A$ and $B$. Moreover, if we know $A = u$, then there can be at most $t$ possible colors for $C$. Thus, we have $h(C|A) \leq \log t$. Similarly, $h(C|B) \leq \log t$. Hence,
\begin{align*}
h(C|B) &\geq I(C; A| B) = h(A|B) - h(A|C,B) = h(A) - I(A; B) - h(A|C,B) \\
&\geq h(A) - I(A;B) - h(A|C) = I(A;C) - I(A; B).
\end{align*}

Thus $h(C) = h(C|A) + I(A; C) \leq h(C|A) + h(C|B) + I(A;B) \leq \log(t^2/p)$.

Notice that $|E| \geq pn^2$. The inequality implies that there must be a color that is used by at least $pn^2/(t^2/p) = (pn/t)^2$ edges. 
\end{proof}

Now, we proceed to prove the Theorem.  

\begin{proof}
For any 3 points $u, v, w \in X$, let $c_{uvw}$ be the barycenter of the triangle $\triangle _{uvw}$. By a simple counting argument, there must be a point $w \in X$ such that for at least $p$ fraction of the ordered pairs $(u, v) \in X \times X$ the triangle $\triangle_{uvw}$ is $\delta$--slim.  We fix such a point $w$ from now on.  Define the bipartite graph $G = (U = X, V = X, E)$ as follows.  For any two vertices $u \in U$ and $v \in V$ there is an edge  $(u, v) \in E$ if and only if the triangle $\triangle_{uvw}$ is $\delta$--slim. The color of $(u,v) \in E$ is $c_{uvw}$.  Then, $|E| \geq pn^2$.  Moreover, if $c$ is the color of $(u,v)$ then $c$ is in the $\delta$--neighbourhood of $[uw]$. Thus, any vertex $u$ can be incident to at most $DM$ colors.  By Lemma~\ref{LL}, there must be a color $c$ that is used by at least $p^2n^2/(DM)^2$ edges in $E$.  Notice that for each such edge $(u, v) \in E$, $c$ is in the $\delta$--neighbourhood of $[uv]$. Thus, $[uv]$ contains a vertex in the ball $B(c, \delta)$.  Since $\card{B(c, \delta)} \leq M$, for some vertex $c' \in B(c, \delta)$ and $p^2n^2/(D^2M^3)$ different ordered pairs $(u, v) \in X \times X$ the geodesic segment $[uv]$ contains $c'$. Therefore, the congestion at the vertex $c'$ is at least $p^2n^2/(D^2M^3)$.
\end{proof}

\section{Congestion on Product of Trees}

Let $T_{d}^{k}$ be the infinite graph defined as the product of $k$ infinite trees of degree $d$. For some fixed integers $d \geq 2$ and $ k\geq 1$. Let us fix an arbitrary vertex $v^*$ in $T_{d}^{k}$ and let $B(v^*,r)$ be the ball of centrer $v^*$ and radius $r$ in this graph. Let $n=n(r)$ be the cardinality of $B(v^*,r)$. It is easy to see that $|B(v^*,r)|=C_{d,k}(d-1)^{rk}$ where $C_{d,k}$ is a constant that depends on $d$ and $k$ and independent on $r$. We are interested in the asymptotic traffic behavior at $v^*$ as $r$ increases. More precisely, we prove the following theorem.

\begin{theorem}
The traffic load at $v^*$ behaves as $\Theta\big(n^2/\log_{d-1}^{k-1}(n)\big)$ as $r$ increases.
\end{theorem}

\begin{proof}

Let $0 \leq r_1, r_2 \leq r$ be two integers. We random select two vertices such that $d(v^*, v_1) = r_1$, $d(v^*, v_2) = r_2$.  Then, randomly select a shortest path between $v_1$ and $v_2$. We consider the probability that the selected path contains $v^*$. 

We shall denote $v_1 = (v_{1,1}, v_{1,2}, \cdots, v_{1,k})$, where $v_{1,i}$ is a vertex in tree $T_d$, denoting the $i$-th component of $v_1$.  Similarly, let $v_2 = (v_{2,1}, v_{2,2}, \cdots, v_{2,k})$. $v^*$ is the point in $T_d^k$ where all components are the root of $T_d$.  $v^*$ is in some shortest path between $v_1$ and $v_2$ iff for every $i \in [k]$, the root of $T_d$ is in the shortest path between $v_{1,i}$ and $v_{2,i}$. 

Let $r_{1, i}$ (resp. $r_{2,i}$) be the depth of $v_{1, i}$ (resp. $v_{2,i}$). Then $\sum_{i = 1}^{k}r_{1,i} = r_1$ and $\sum_{i = 1}^{k}r_{2,i} = r_2$. We say $v_1$ has pattern $\vec{r_1}=(r_{1,1}, r_{1,2}, \cdots, r_{1,k})$ and $v_2$ has pattern $\vec{r_2}=(r_{2,1}, r_{2,2}, \cdots, r_{2,k})$. Under the condition that the patterns of $v_1$ and $v_2$ are  $\vec{r_1}$ and  $\vec{r_2}$ respectively, the probability that $v^*$ is in some shortest path between $v_1$ and $v_2$ is at least $(1-1/d)^k = \Theta(1)$. (More precisely, it is exactly $(1-1/d)^t$, where $t$ is the number of integers $i$ such that $r_{1,i} > 0$ and $r_{2, i} > 0$.)   

For fixed $\vec{r_1}$, the number of vertices $v_1$ such chat $d(v^*, v_1) = r_1$ is exactly $d^{r_{1,1}}d^{r_{1,2}} \cdots d^{r_{1,k}} = d^{r_1}$. Let $M_1 = \Theta(r_1^{k-1})$ be the number of different patterns $\vec{r_1}=(r_{1,1}, r_{1,2}, \cdots, r_{1,k})$ such that $r_{1,1} + r_{1,2} + \cdots + r_{1, k} = r_1$. Thus, the probability that $v_1$ has pattern $\vec{r_1}$ is exactly $1/M_1$.  Define $M_2$ similarly.  Then, for fixed $\vec{r_1}$ and $\vec{r_2}$, the probability that $v_1$ and $v_2$ has patterns $\vec{r_1}$ and $\vec{r_2}$ respectively is $1/(M_1M_2)$. Conditioned on this event, the probability that $v^*$ is in some shortest path between $v_1$ and $v_2$ is $\Theta(1)$.  Conditioned on this event, the probability that the selected shortest path contain $v^*$ is $ f(\vec{r_1} + \vec{r_2}, \vec{r_1})$, where $f(\vec{a}, \vec{r_1})$ is the probability that a random shortest path from $(0,0, \cdots, 0)$ to $\vec{a}$ in the $k$-dimensional grid contains point $\vec{r_1}$.

Consider all patterns $\vec{r_1}$ and $\vec{r_2}$, the probability that $v^*$ is in the selected shortest path is 

\begin{align*}
&\sum_{\vec{r_1}, \vec{r_2}:*}\frac{1}{M_1M_2} \Theta(1)f(\vec{r_1} + \vec{r_2}, \vec{r_1}) = \frac{\Theta(1)}{M_1M_2}\sum_{\vec{a}:\#}\sum_{\vec{r_1}:\&} f(\vec{a}, \vec{r_1}) \\
&= \frac{\Theta(1)}{M_1M_2}\sum_{\vec{a}:\#} 1 = \Theta\left(\frac{(r_1+r_2)^{k-1}}{r_1^{k-1}r_2^{k-1}}\right) = \Theta\left(\frac{1}{\min\set{r_1,r_2}^{k-1}}\right).
\end{align*}

Now, we consider all possible combinations of $r_1$ and $r_2$. If we randomly select a vertex $v_1$ in the ball $B(v^*, r)$,  the probability that $d(v^*, v_1) = r_1$ is $\Theta\left(\frac{r_1^{k-1}d^{r_1}}{r^{k-1}d^r}\right)$. Thus, the final probability that $v^*$ is in the selected shortest path is 
\begin{align*}
&\sum_{r_1 \leq r, r_2 \leq r}\Theta\left(\frac{r_1^{k-1}d^{r_1}}{r^{k-1}d^r}\times \frac{r_2^{k-1}d^{r_2}}{r^{k-1}d^r}\times \frac{1}{\min\set{r_1,r_2}^{k-1}}\right)\\
&=\frac{1}{r^{2k-2}d^{2r}}\Theta\left(\sum_{r_1,r_2:0\leq r_1 \leq r_2 \leq r}\frac{(r_1r_2)^{k-1}d^{r_1+r_2}}{r_1^{k-1}}\right)\\
&=\frac{1}{r^{2k-2}d^{2r}}\Theta\left(\sum_{r_1:0\leq r_1\leq r}r^{k-1}d^{r_1+r}\right) \\
&= \frac{1}{r^{2k-2}d^{2r}}\Theta\left(r^{k-1}d^{2r}\right) = \Theta\left(\frac{1}{r^{k-1}}\right).
\end{align*}

Congestion of $v^*$ is $\Theta(d^{2r}r^{k-1})$. Thus, fraction = $1/\Theta(r^{k-1}) = 1/\log_{d-1}^{k-1}(n)$.
\end{proof}

Traffic congestion in product of trees with different degrees is easier to analyze. For instance, it is not difficult to see that the maximum vertex congestion under geodesic routing for $T_{d_1}\times T_{d_2}$ where $3\leq d_1<d_2$ behaves as $\Theta(n^2)$ since the larger tree dominates the other due to the exponential growth. More generally, the following result is true.

\begin{theorem}
Let $p\geq 1$, $3\leq d_1<d_2<\ldots<d_{p}$ and $\{k_{i}\}_{i=1}^{p}$ be a sequence of positive integers. Let $X=\prod_{i=1}^{p}T_{d_i}^{k_i}$. Then 
the traffic load at $v^*$ behaves as $\Theta\big(n^2/\log_{d_p-1}^{k_p-1}(n)\big)$ as $r$ increases.
\end{theorem}

\section{Traffic on Expanders}\label{expanders}

In this Section we construct a constant degree family of expanders with $\Theta(n^2)$ congestion. This result is in contrast to random regular graphs that have congestion $O(n\log^{3}(n))$ (as shown in \cite{T}).

\begin{theorem}
There exists a family $\{G_{n}\}_{n=1}^{\infty}$ of constant--degree expanders with congestion $\Theta(n^2)$.
\end{theorem}

\subsection{Construction of the Expander Graph}
For an even integer $h$, let $T$ be a tree of depth $h$ defined as follows.  Each node in depth $0$ to $h/2-1$ of $T$ has $3$ children, and each node in depth $h/2$ to $h-1$ has $2$ children (root has depth $0$ and leaves have depth $h$). Thus, $T$ has exactly $n := 3^{h/2} \times 2^{h/2}=6^{h/2}$ leaves.  We use $L(T)$ to denote the set of leaves of $T$.  Define  $\lambda(d)$ to be the number of leaves in a sub--tree of $T$ rooted at some vertex of depth $h-d$ . Then, we have
\[\lambda(d) = \begin{cases}
2^d & 0 \leq 	d \leq h/2 \\
2^{h/2}3^{d-h/2}& h/2 < d \leq h
\end{cases}.
\]

Construct a graph $G$ as follows (see Figure~\ref{fig:constructing-expander}). Let $A$ be a degree--3 expander of size $2n$ and expansion constant $\alpha$.  Let $T_\sfL$ and $T_\sfR$ be two copies of  the tree $T$.  We create a random matching between the $2n$ leaves of $T_\sfL$ and $T_\sfR$ and the vertices of $A$. We add a \emph{matching edge} between each matched pair of vertices. We also add a vertex $v^*$ that is connected to the roots of $T_\sfL$ and $T_\sfR$. Finally, we scale the matching edges and expander edges by a factor of $c$ (i.e., replace those edges with paths of length $c$) for some constant even number $c$ to be determined later. 

\begin{figure}
\centering
\includegraphics[scale=0.4]{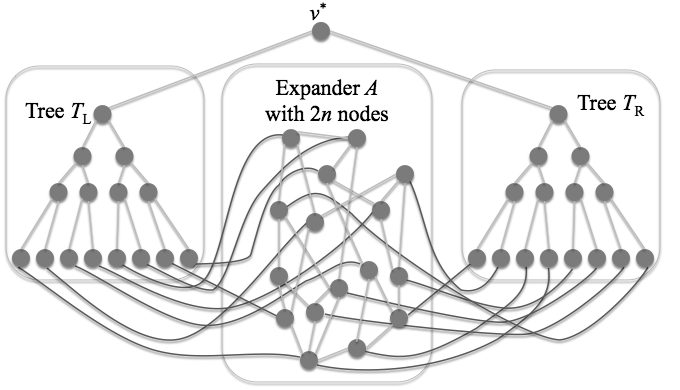}
\caption{Construction of the expander $G$. There are two trees $T_\sfL$ and $T_\sfR$, each with $n$ leaves, a root node $v^*$ and an expander $A$ with $2n$ nodes and expansion constant $\alpha$.  We connect $v^*$ to the two roots of $T_\sfL$ and $T_\sfR$. We create a random matching between the $2n$ leaves of $T_\sfL$ and $T_\sfR$ and the $2n$ vertices of $A$; there is a path of length $c$ connecting each pair in the matching. We also replace edges of $A$ with paths of length $c$.}
\label{fig:constructing-expander}
\end{figure}

\begin{prop}
The graph $G$ is an expander of degree at most $4$ and size $\Theta(n)$.
\end{prop}

\begin{proof}
Since replacing edges with paths of length $c$ only decrease the expansion by a factor of $c$, we only need to consider the graph obtained before the scaling operation.  Let $G'$ be that graph.  For the simplicity of the notation, let $T'$ be the tree rooted at $v^*$ with 2 sub--trees $T_\sfL$ and $T_\sfR$ and $L(T') = L(T_\sfL) \cup L(T_\sfR)$. We also use $I(T')$ to denote the set of inner vertices of $T'$. Assume that $G'$ is not an $0.01\alpha$--expander. Then, let $S \subseteq V(G')$ be a set of size at most $\card{V(G')}/2$ such that $E_{G'}(S, V(G') \setminus S) < 0.01\alpha\card{S}$.

We notice the following two facts: 
\begin{enumerate}
\item If $\card{S \cap I(T')} \geq \card{S \cap L(T')}+s$ then there must be at least $s$ edges between $S \cap V(T')$ and $V(T') \setminus S$.
\vspace{0.3cm}
\item If $\big|\card{S \cap L(T')} - \card{S \cap V(A)} \big| \geq s$ then there must be at least $s$ matching edges between $S$ and $V(G')\setminus S$.
\end{enumerate}
By the first fact, we can assume $\card{S \cap I(T')} \leq 0.6|S|$. Then 
$$
\card{S \cap L(T')} + \card{S \cap V(A)} \geq 0.4|S|.
$$ 
By the second fact, we have $\card{S \cap V(A)} \geq 0.15|S|$. If $\card{S \cap V(A)} \leq \card{V(A)}/2$ then 
$$
E_{G'}(S, V(G')\setminus S) \geq 0.15\alpha|S|.
$$ 
Thus, we have $\card{S \cap V(A)} \geq \card{V(A)}/2$. If $\card{S \cap V(A)} \leq 0.9\card{V(A)}$ then 
$$
E_{G'}(S, V(G') \setminus S) \geq \alpha\card{V(A) \setminus S} \geq \frac{0.1}{0.9}\alpha\card{S \cap V(A)}\geq \alpha|S|/60.
$$
Thus, we have $\card{S \cap V(A)} \geq 0.9\card{V(A)}$, which implies $\card{S \cap L(T')} \geq 0.9\card{V(A)} - 0.1|S|$ by the second fact. Then $|S| \geq 1.8\card{V(A)}-0.1|S|$, implying that 
$$
|S| \geq 1.6\card{V(A)} \geq \frac{1.6}{3}\card{V(G')} > 0.5 \card{V(G')}
$$
which is a contradiction. 
\end{proof}

\subsection{Proof of the High Congestion in $G$}

Consider the set $L(T_\sfL) \times L(T_\sfR)$ of $n^2$ pairs. We shall show that for a constant fraction of pairs $(u,v)$ in this set, the shortest path connecting $u$ and $v$ will contain $r$.  It is easy to see that for every $(u, v) \in L(T_\sfL) \times L(T_\sfR)$, there is a path of length $2h + 2$ connecting $u$ and $v$ that goes through $v^*$.

Focus on the graph  $G\setminus v^*$. We are interested in the number of pairs $(u,v) \in L(T_\sfL) \times L(T_\sfR)$ such that $d_{G\setminus v^*}(u,v) \leq 2h + 2$.  Fix a vertex $u \in L(T_\sfL)$ from now on. Consider the set $\mP$ of simple paths starting at $u$ and ending at $L(T_\sfR)$. We say a path $P \in \mP$ has rank $r$ if it enters and leaves the expander $r$ times.  Notice that we always have $r \geq 1$, since we must use the expander $A$ from $u$ to $L(T_\sfR)$.

For a path $P$ of rank $r$, we define the \emph{pattern} of $P$, denoted by $\ptn(P)$, as a sequence $t = (t_1, t_2, \cdots, t_{2r+1})$ of $2r + 1$ non-negative integers as follows. For $1 \leq i \leq r$, $ct_{2i}$ is the length of the sub-path of $P$ correspondent to the $i$-th traversal of $P$ in the expander $A$. (Recall that we replaced each expander edge with a path of length $c$.)   The path $P$ will contain $r+1$ sub-paths in the two copies of $T$ (the first and/or the last sub-path might have length 0). Let $2t_{2i-1}$ be the length of the $i$-th sub-path in the tree. Notice that $P$ can only enter and leave the trees through leaves and thus the lengths of those sub-paths must be even.  If some path $P \in \mP$ has $\ptn(P) = (t_0, t_1, \cdots, t_{2r+1})$, then the length of $P$ is exactly $\len(t_{[2r+1]}) := 2\sum_{i=0}^{r}t_{2i+1} + c\sum_{i=1}^{r}(t_{2i}+2)$, where $t_{[2r+1]}$ denotes the sequence $(t_1, t_2, \cdots, t_{2r+1})$. Notice that the $c(d_{2i}+2)$ term in the definition of $\len$ includes the $2c$ edges replaced by the 2 matching edges through which the path enters and leaves the expander.

We call a sequence $(t_0, t_1, \cdots, t_{2r+1})$ of non-negative integers \emph{a valid pattern} of rank $r$ if $\len(t_{[2r+1]}) \leq 2h+2$.  

\begin{lemma}
\label{lemma:num-paths-for-a-pattern}
The number of paths $P \in \mP$ with $\ptn(P) = t_{[2r+1]}$ is at most $$\left(\frac{3}{2}\right)^{r}\left[\prod_{i = 1}^{r}\lambda(t_{2i-1})2^{t_{2i}}\right]\lambda(t_{2r+1}).$$
\end{lemma}
\begin{proof}
For any leaf $v$ in $T$, we have at most $\lambda(t)$ possible simple paths of length $2t$ in $T$ that start at $v$ and end at $L(T)$. For a degree-$3$ graph(in particular, a degree-3 expander), we have at most $\frac32\times 2^t$ simple paths that start at any fixed vertex. From a leaf in $T_\sfL$ or $T_\sfR$, we only have one way to enter the expander. Similarly, we only have one way to leave the expander from a vertex in the expander. Thus, the total number of simple paths of pattern $t_{[2r+1]}$ is at most 
\[
\lambda(t_1) \left(\frac 32 2^{t_2}\right)\lambda(t_3) \left(\frac 32 2^{t_4}\right)\cdots \lambda(t_{2r+1})=\left(\frac{3}{2}\right)^{r}\left[\prod_{i = 1}^{r}\lambda(d_{2i-1})2^{d_{2i}}\right]\lambda(t_{2r+1}).
\]
\end{proof}

We fix the rank $r \geq 1$ from now on.   For some integer $\ell \in [0, r]$, suppose $t_{[2\ell]}$ is a prefix of some valid pattern of rank $r$. Define $W(t_{[2\ell]})$ to be the maximum $t_{2\ell+1}$ such that $t_{[2\ell+1]}$ is a prefix of some valid pattern of rank $r$.  That is, 
\[
W\left(t_{[2\ell]}\right)=
\floor{
\frac{2h+2 - 2\sum_{i=1}^{\ell}t_{2i-1} - c\sum_{i = 1}^{\ell}(t_{2i}+2) -2c(r-l)}{2}
},
\]
Similarly,  we define 
\[
W\left(t_{[2\ell - 1]}\right) = 
\floor{\frac{2h+2- 2\sum_{i=1}^{\ell}t_{2i-1} - c\sum_{i = 1}^{\ell-1}(t_{2i}+2) - 2c(r-l)}{c}
}-2.
\]

For some prefix $t_{[2\ell]}$ of some valid pattern of rank $r$, let $\mP_{t_{[2\ell]}}$ be the set of paths $P \in \mP$ of rank $r$ such that $t_{[2\ell]}$ is a prefix of $\ptn(P)$. We prove that 

\begin{lemma} 
\label{lemma:num-paths-for-prefix}
For any $0 \leq l \leq r$, we have 
\[
\card{\mP_{t_{[2\ell]})}} \leq 2C^{r-\ell}\left(\frac32\right)^\ell \prod_{i=1}^{l}\left[\lambda(t_{2i-1})2^{t_{2i}}\right]\times \sqrt{6}^{W\left(t_{[2\ell]}\right)}
\]
for some large enough constant $C$. 
\end{lemma}

\begin{proof}
For $l = r$,   Lemma~\ref{lemma:num-paths-for-a-pattern} implies 
\begin{align*}
\card{\mP_{t_{[2r]}}}  &\leq  \sum_{t_{2r+1}=0}^{W\left(t_{[2r]}\right)}\left(\frac32\right)^r\prod_{i = 1}^{r}\left[\lambda(t_{2i-1})2^{t_{2i}}\right]\lambda(t_{2r+1})\\
&=\left(\frac32\right)^r\prod_{i = 1}^{r}\left[\lambda(t_{2i-1})2^{t_{2i}}\right]
\sum_{t_{2r+1}=0}^{W\left(t_{[2r]}\right)}\lambda(t_{2r+1})\\
&\leq  2\left(\frac32\right)^r\prod_{i = 1}^{r}\left[\lambda(t_{2i-1})2^{t_{2i}}\right]\times \sqrt{6}^{W\left(t_{[2r]}\right)}.
\end{align*}
The last inequality holds since $\sum_{t_{2r+1}=0}^{W\left(t_{[2r]}\right)}\lambda(t_{2r+1}) \leq 2\lambda\left(W\left(t_{[2r]}\right)\right) \leq 2\sqrt{6}^{W\left(t_{[2r]}\right)}$.

Now, suppose the lemma holds for some $1 \leq \ell \leq r$ and we shall prove that it holds for $\ell - 1$.  By the induction hypothesis, we have 
\begin{align*}
\card{\mP_{t_{[2\ell-2]}}} &\leq \sum_{t_{2\ell-1}=0}^{W\left(t_{[2\ell-2]}\right)}
\sum_{t_{2\ell}=0}^{W\left(t_{[2\ell-1]}\right)}2C^{r-\ell}\left(\frac32\right)^\ell\prod_{i=1}^\ell\left[\lambda(t_{2i-1})2^{t_{2i}}\right] \times \sqrt{6}^{W\left(t_{[2\ell]}\right)}\\
&=2C^{r-\ell}\left(\frac32\right)^\ell\prod_{i=1}^{\ell-1}\left[\lambda(t_{2i-1})2^{t_{2i}}\right]
 \sum_{t_{2\ell-1}=0}^{W\left(t_{[2\ell-2]}\right)} \lambda(t_{2\ell-1})
 \sum_{t_{2\ell}=0}^{W\left(t_{[2\ell-1]}\right)}2^{t_{2\ell}}\sqrt{6}^{W\left(t_{[2\ell]}\right)}.
\end{align*}
It is sufficient to prove that $\sum_{t_{2\ell-1}=0}^{W\left(t_{[2\ell-2]}\right)} \lambda(t_{2\ell-1})
 \sum_{t_{2\ell}=0}^{W\left(t_{[2\ell-1]}\right)}2^{t_{2\ell}}\sqrt{6}^{W\left(t_{[2\ell]}\right)} \leq \frac{2C}{3}\sqrt{6}^{W\left(t_{[2\ell-2]}\right)}$.
\begin{align}
\text{LHS}&\leq \frac{1}{1-2/\sqrt{6}^{c/2}} \sum_{t_{2\ell-1}=0}^{W\left(t_{[2\ell-2]}\right)}\lambda(t_{2\ell-1}) \sqrt{6}^{W\left(t_{[2\ell-1]}\Join(0)\right)} \label{inequality:induction-first}\\
& \leq \frac{1}{1-2/\sqrt{6}^{c/2}}\left(\frac{\sqrt{6}^{W\left(t_{[2\ell-2]}\Join(0,0)\right)}}{1-\sqrt{2/3}} + \frac{\lambda\left(W\left(t_{[2\ell-2]}\right)\right)}{1-\sqrt{2/3}}\right) \label{inequality:induction-second}\\
& \leq \frac{1}{1-2/\sqrt{6}^{c/2}}\frac{2}{1-\sqrt{2/3}}\sqrt{6}^{W\left(t_{[2\ell-2]}\right)}, \label{inequality:induction-third}
\end{align}
where $t_{[2\ell-1]}\Join (1)$ denotes the sequence obtained by concatenating the two sequences $t_{[2\ell-1]}$ and $(1)$.

We explain Inequalities~\eqref{inequality:induction-first},\eqref{inequality:induction-second} and \eqref{inequality:induction-third} one by one.  Focus on the term $Q=2^{t_{2\ell}}\sqrt{6}^{W\left(t_{[2\ell]}\right)}$. If we increase $t_{2\ell}$ by 1, then $W\left(t_{[2\ell]}\right)$ will decrease by exactly $c/2$, by the definition of $W$. (We assumed $c$ is an even number.) Thus, $Q$ will decrease by a factor of $\sqrt{6}^{c/2}/2$.  By the rule of the geometric sum, 
\[ \sum_{t_{2\ell}=0}^{W\left(t_{[2\ell-1]}\right)}Q \leq \frac{ Q|_{t_{2\ell}=0}}{1-2/\sqrt{6}^{c/2}} = \frac{\sqrt{6}^{W\left(t_{[2\ell-1]}\Join (0)\right)}}{1-2/\sqrt{6}^{c/2}},
\]
implying Inequality~\eqref{inequality:induction-first}. 

Now focus on the term $Q=\lambda(t_{2\ell-1})\sqrt{6}^{W\left(t_{[2\ell-1]}\Join (0)\right)}$. If we increase $t_{2\ell - 1}$ by 1, then $W\left(t_{[2\ell-1]}\Join (0)\right)$ will decrease by 1. Then, $Q$ will either decrease by a factor of  $\sqrt{6}/2 = \sqrt{3/2}$, or increase by a factor of $3/\sqrt{6}=\sqrt{3/2}$, depending on whether $t_{2\ell-1}\leq h/2$. We can split the sum $\displaystyle \sum_{t_{2\ell-1}=0}^{W(t_{[2\ell-2]})}Q$ into 2 sums at the point $h/2$ if necessary.  Again, using the geometric sum, we have
\[
\sum_{t_{2\ell-1}=0}^{W(t_{[2\ell-2]})}Q \leq \frac{Q|_{t_{2\ell-1} = 0}}{1-\sqrt{2/3}} + \frac{Q|_{t_{2\ell-1}=W(t_{[2\ell-2]})}}{1-\sqrt{2/3}},
\]
implying Inequality~\eqref{inequality:induction-second}.

Inequality~\eqref{inequality:induction-third} follows from the fact that $W\left(t_{[2\ell-2]}\Join(0, 0) \right)=W\left(t_{[2\ell-2]} \right)$ and $\lambda(t) \leq \sqrt{6}^t$ for any $t \in [0, h]$.

This finishes the proof if we let $\displaystyle C=\frac{3}{\left(1-2/\sqrt{6}^{c/2}\right)\left(1-\sqrt{2/3}\right)}$.
\end{proof}

Notice that $W(\emptyset) = \floor{\frac{2h+2 - 2rc }{2}} = h + 1 - rc$ for an even integer $c$. Thus, $$\card{\mP_{\emptyset}} \leq 2C^r\sqrt{6}^{h+1-rc}=2\sqrt{6}\left(\frac{C}{\sqrt{6}^c}\right)^r\sqrt{6}^h =2\sqrt{6}\left(\frac{C}{\sqrt{6}^c}\right)^rn .$$

$\mP_\emptyset$ is  essentially the set of paths in $\mP$ with rank $r$. For a large enough $c$, we have $C < \sqrt{6}^c$. Summing up over all $r \geq 1$, we have that the total number of paths in $\mP$ is at most 
\[
\frac{2\sqrt{6}C/\sqrt{6}^c}{1-C/\sqrt{6}^c}n.
\]

For a large enough constant $c$ (say, $c = 6$), the number will be at most $n/2$. 

Then, consider the graph $G$. We know there is a path of length $2h+2$ from $u$ to $v$ via $v^*$ for any vertex $v \in L(T_\sfR)$. Thus, for a fixed vertex $u \in L(T_\sfL)$,  there are at least $n/2$ vertices $v \in L(T_\sfR)$ such that the shortest path between $u$ and $v$ contains $v^*$. Therefore,  there are $n^2/2$ pairs $(u,v)$  such that the shortest path between $u$ and $v$ contain $v^*$.

{\it Acknowledgement:} This work was supported by NIST Grant No. 60NANB10D128.

\end{document}